\documentclass[twoside,12pt]{article}

\usepackage{amsmath,amssymb,amsthm,mathrsfs}
\usepackage{enumitem}
\usepackage{xcolor}
\usepackage[top=2cm,bottom=2cm,left=2cm,right=2cm]{geometry}
\usepackage[colorlinks,linkcolor=blue,citecolor=blue,urlcolor=blue]{hyperref}

\allowdisplaybreaks[4]
\setlist[enumerate,1]{label=(\arabic*)}

\newcommand{\Q}{\mathbb{Q}}
\newcommand{\R}{\mathbb{R}}
\newcommand{\Z}{\mathbb{Z}}
\newcommand{\ord}{\mathrm{ord}}

\newcommand{\rand}{\mathrm{rand}}
\newcommand{\Mord}{M^{\ord}}

\newcommand{\Prob}{\mathbb{P}}
\DeclareMathOperator{\E}{\mathbb{E}}

\newtheorem{theorem}{Theorem}[section]
\newtheorem{conjecture}{Conjecture}
\newtheorem{lemma}{Lemma}[section]

\newtheorem{corollary}[theorem]{Corollary}
\newtheorem{proposition}{Proposition}[section]

\begin{document}

\title{\textbf{Complete Resolution of the Butler--Costello--Graham Conjecture on Monochromatic Constellations}}

\author{Gang Yang\footnote{Graduate School of Environment and Information Sciences, Yokohama
National University, 79-2 Tokiwadai, Hodogaya-ku, Yokohama 240-8501,
Japan. {\tt gangyang98@outlook.com}},\quad
Yaping Mao\footnote{Corresponding author: Academy of Plateau
Science and Sustainability, and School of Mathematics and Statistics, Qinghai Normal University, Xining, Qinghai 810008, China. {\tt yapingmao@outlook.com; myp@qhnu.edu.cn}}}
\date{}
\maketitle

\begin{abstract}
A constellation pattern is a finite increasing rational sequence
\(Q=[0=q_0<q_1<\cdots<q_k=1]\), and a \(Q\)-constellation in
\([n]\) is obtained by scaling and translating a rational pattern $Q$, with key examples including arithmetic progressions. In 2010,
Butler, Costello, and Graham proposed a conjecture, that is, 
for any constellation pattern $Q$ there is a coloring pattern of $[n]$ that has $\gamma n^2+o\left(n^2\right)$ monochromatic constellations, where $\gamma$ is smaller than the coefficient for a random coloring. In this paper, we confirm this conjecture. As applications of this conjecture, we obtain interval-uncommon translation-invariant linear systems associated with rational constellations and a ground-state bound for deterministic arithmetic hypergraph spin systems. 
\\[2mm]
\end{abstract}

\section{Introduction}\label{sec:introduction}

Enumerative combinatorics studies quantitative questions about finite
discrete structures: rather than merely asking whether a prescribed object
exists, one asks how many such objects occur \cite{Stanley99}.  Classical
examples include the enumeration of orbits under group actions, as formalized
by P\'olya's enumeration theorem \cite{Polya37}, the enumeration of formal
languages in the Chomsky--Sch\"utzenberger theory \cite{ChomskySchutzenberger63},
and the enumeration of spanning trees, originating in Kirchhoff's work on
electrical networks and later developed in combinatorics, chemistry, physics,
and network science \cite{Kirchhoff47,Mallion75,McKay83,Lyons05,ZLWZ11}.
These examples illustrate a general theme: qualitative structural questions
often lead naturally to quantitative counting problems.

Ramsey theory is another major source of such quantitative problems.  Its
classical form asserts that sufficiently large colored structures necessarily
contain monochromatic substructures of a prescribed type \cite{GRS90,LR04}.  In
the integer setting, this philosophy includes monochromatic arithmetic
progressions, polynomial extensions of van der Waerden's theorem, Rado-type
partition theorems, Hales--Jewett theory, and density and discrepancy questions
for arithmetic configurations \cite{BL96,BL99,CLP21,FK06,MS96,Po12,RR95,RR97,Sh88}.
The passage from existence to counting leads to \emph{Ramsey multiplicity
theory}: given a fixed pattern and a class of colorings, one asks for the
minimum possible number of monochromatic copies of that pattern
\cite{BR80,GRR96,JST96}.

In graph Ramsey theory this viewpoint led to the notion of common and uncommon
graphs: a graph is common if the random edge-coloring asymptotically minimizes
the number of monochromatic copies \cite{Goodman59,BR80,JST96}.  Goodman's
theorem shows that triangles are common \cite{Goodman59}, while Erd\H{o}s
conjectured analogous behavior for complete graphs and Burr--Rosta proposed a
broad extension to all graphs \cite{Erdos62,BR80}.  These conjectures were
disproved by Sidorenko and Thomason, and Jagger, \v{S}\v{t}ov\'\i\v{c}ek, and
Thomason later showed that every graph containing a copy of \(K_4\) is
uncommon \cite{Sidorenko89,Th89,JST96}.  Recent work continues to show that
Ramsey multiplicity constants can display highly non-random extremal behavior
\cite{FoxWigderson23}.

Ramsey multiplicity has found applications in a wide range of additive settings, including integer intervals \cite{CE23,GRR96}, finite abelian groups \cite{SW17,Ver23}, and finite vector spaces  \cite{FPZ21,RS26}.
A prototypical example, posed by Graham, R\"odl, and Ruci\'nski, is to
determine the minimum number of monochromatic solutions to \(x+y=z\) in a
two-coloring of \([n]=\{1,2,\ldots,n\}\) \cite{GRR96}.  This problem was
solved independently by Robertson and Zeilberger and by Schoen, with a later
proof by Datskovsky; under the unordered convention the asymptotic minimum is
\(n^2/22+O(n)\), equivalently \(n^2/11+O(n)\) in the ordered convention
\cite{Da03,RZ98,Sc99}.  Subsequent work treated three-term arithmetic
progressions and generalized Schur-type equations, showing that even very
simple-looking additive configurations can have extremal colorings that differ
substantially from random colorings \cite{PRS08,Th09,TW17}.  More recently,
the commonness problem for linear equations and systems has been studied in
finite-group and finite-field models, further emphasizing the connection
between additive combinatorics and Ramsey multiplicity \cite{CE23,FPZ21,RS26,SW17,Ver23}.

Butler, Costello, and Graham \cite{BCG10} introduced the notion of a \emph{constellation} to unify many rational affine patterns.  Here and throughout, \(\Q\), \(\R\), and
\(\Z\) denote the sets of rational, real, and integer numbers, respectively. A
\textit{constellation pattern} is a strictly increasing rational sequence
\[
Q=[q_0=0<q_1<\cdots<q_k=1],
\qquad q_i\in\Q.
\]
A \emph{\(Q\)-constellation in \([n]\)} is a homothetic copy of this pattern
whose points are integers:
\[
\{s+q_0d,s+q_1d,\dots,s+q_kd\}\subseteq [n],
\qquad s,d\in\Z,\quad d\ne 0.
\]
We call \(s\) and \(s+d\) the \emph{endpoints} of the \(Q\)-constellation in
\([n]\).  Arithmetic progressions correspond to the choice \(q_i=i/k\).  We
say that \(Q\) is \emph{symmetric} if \(q_i+q_{k-i}=1\) for all
\(0\le i\le k\), and \emph{nonsymmetric} otherwise.

Given a two-coloring \(\chi:[n]\to\{\pm1\}\), let \(M_Q(\chi)\) denote the
number of monochromatic \(Q\)-constellations in \([n]\).  Let \(\chi_{\rand}\)
be the uniformly random two-coloring of \([n]\), in which the colors of the
points \(1,\dots,n\) are independent and uniformly distributed on
\(\{\pm1\}\).  A fixed \(Q\)-constellation in \([n]\) has \(k+1\) distinct
points, and hence is monochromatic under \(\chi_{\rand}\) with probability
\(2\cdot 2^{-(k+1)}=2^{-k}\).  Writing \(D\) for the least common denominator
of \(q_0,\dots,q_k\), the total number of \(Q\)-constellations in \([n]\) is
\(n^2/D+O_Q(n)\) in the nonsymmetric case and \(n^2/(2D)+O_Q(n)\) in the
symmetric case, as justified in Subsection~\ref{subsec:total-counts}.  Therefore, $\E\!\left[M_Q(\chi_{\rand})\right]
=
\gamma_{\rand}(Q)n^2+O_Q(n)$,
where
\[
\gamma_{\rand}(Q)=
\begin{cases}
\dfrac{1}{2^kD}, & Q \text{ is nonsymmetric},\\[8pt]
\dfrac{1}{2^{k+1}D}, & Q \text{ is symmetric}.
\end{cases}
\]
Here and throughout, \(\E\) denotes expectation. We call \(\gamma_{\rand}(Q)\) the random-coloring coefficient.

For three-point patterns \(Q=[0,q,1]\), Butler, Costello, and Graham proved
that the random-coloring coefficient can always be improved \cite{BCG10}.  They
left open the corresponding problem for constellation patterns with four or
more points.  The two-point pattern is the only trivial exception.  Indeed, if
\(Q=[0,1]\) and one color class has size \(r\), then the number of
monochromatic \(Q\)-constellations is
\[
\binom r2+\binom{n-r}{2}
=
\frac{r^2+(n-r)^2-n}{2}
\ge
\frac{n^2}{4}-\frac n2.
\]
Thus the random-coloring coefficient \(1/4\) is already optimal for
\(Q=[0,1]\), and no strict improvement is possible.

It is therefore natural to formulate the Butler--Costello--Graham problem as
follows.

\begin{conjecture}[Butler--Costello--Graham problem, \cite{BCG10}]
\label{con}
For every rational constellation pattern
\[
Q=[q_0=0<q_1<\cdots<q_k=1]
\]
with \(k\ge 2\), there exists a sequence of two-colorings
\(\chi_n:[n]\to\{\pm1\}\) such that
\[
M_Q(\chi_n)\le \gamma n^2+o(n^2)
\]
for some constant \(\gamma<\gamma_{\rand}(Q)\).
\end{conjecture}

Our main theorem confirms Conjecture~\ref{con}.

\begin{theorem}\label{thm:main}
Let \(Q=[q_0=0<q_1<\cdots<q_k=1]\) with \(q_i\in\Q\) and \(k\ge 2\).  Then
there exist a constant \(\delta_Q>0\) and a sequence of two-colorings
\(\chi_n:[n]\to\{\pm1\}\) such that
\[
M_Q(\chi_n)
\le
\bigl(\gamma_{\rand}(Q)-\delta_Q\bigr)n^2+O_Q(n),
\]
where
\[
\gamma_{\rand}(Q)=
\begin{cases}
\dfrac{1}{2^kD}, & Q \text{ is nonsymmetric},\\[8pt]
\dfrac{1}{2^{k+1}D}, & Q \text{ is symmetric},
\end{cases}
\]
and \(D\) is the least common denominator of \(q_0,\dots,q_k\).
\end{theorem}

Our proof follows a unified strategy rather than relying on isolated
constructions tailored to particular patterns.  We first encode monochromatic
counts by a continuous functional \(\Phi_Q(b)\) associated with a bias function
\(b:[0,1]\to[-1,1]\).  The key discrete-to-continuous step is a uniform
Riemann-sum estimate on each residue class modulo \(D\), valid for every
Lipschitz function and every residue class, with a uniform \(O(1)\) remainder.
After transferring the discrete counting problem to the continuous functional,
we expand \(\Phi_Q(\lambda g)\) around the random point \(b\equiv0\).  The
first nontrivial term is a quadratic form \(T_2(g)\).  We then show that the
Fourier mode \(u(x)=\cos(2\pi D x)\) is a zero direction for \(T_2\), while
its mixed bilinear term with \(v(x)=\cos(2\pi(D+1)x)\) is strictly positive.
Hence a small perturbation of \(u\) in the direction \(-v\) yields
\(T_2(g)<0\), and for sufficiently small \(\lambda>0\) gives
\(\Phi_Q(\lambda g)<2^{-k}\).  A final application of the method of conditional
expectations converts the resulting biased random construction into a
deterministic coloring.

The paper is organized as follows.  Section~\ref{sec:preliminaries} develops
the endpoint model and the basic counting lemmas.  Section~\ref{sec:continuous-transfer}
introduces the continuous functional, proves the discrete-to-continuous
transfer, and derives the second-order expansion at the random point.
Section~\ref{sec:negative-direction} constructs an explicit negative direction
for the quadratic form.  Section~\ref{sec:deterministic-colorings} converts the
biased random construction into deterministic colorings and completes the proof
of Theorem~\ref{thm:main}.  Section~\ref{sec:applications} presents
applications of the main result.

\section{Preliminaries}\label{sec:preliminaries}

Throughout this section, fix a rational pattern
\(Q=[q_0=0<q_1<\cdots<q_k=1]\), where \(q_i\in\Q\), and assume \(k\ge2\).
Recall that \(D\) is the least common denominator of \(q_0,\dots,q_k\), and
write \(a_i=Dq_i\) for \(0\le i\le k\).  Then
\(0=a_0<a_1<\cdots<a_k=D\).  By the minimality of \(D\), we have
\(\gcd(a_1,\dots,a_{k-1},D)=1\).

The next proposition shows that every \(Q\)-constellation in \([n]\) can be
parametrized by its two endpoints \(p,q\), as in \textnormal{(ii)} below; this
parametrization will be used throughout.

\begin{proposition}\label{pro-1}
The following two descriptions of a \(Q\)-constellation in \([n]\) are
equivalent:
\begin{enumerate}[label=\textnormal{(\roman*)}]
\item a set of the form
\(\{s+q_0d,s+q_1d,\dots,s+q_kd\}\subseteq[n]\), where
\(s,d\in\Z\) and \(d\ne0\);
\item a set of the form
\begin{equation}\label{eq-xpq}
\left\{
x_i(p,q)=\frac{(D-a_i)p+a_iq}{D}:0\le i\le k
\right\}\subseteq[n],
\end{equation}
where \(p,q\in[n]\), \(p\ne q\), and \(p\equiv q\pmod D\).
\end{enumerate}
\end{proposition}

\begin{proof}
Suppose first that
\(X=\{s+q_0d,s+q_1d,\dots,s+q_kd\}\subseteq[n]\), where
\(s,d\in\Z\) and \(d\ne0\).  Set \(p=s\) and \(q=s+d\).  Since \(q_0=0\)
and \(q_k=1\), we have \(p,q\in[n]\), and \(p\ne q\).  It remains to show
that \(p\equiv q\pmod D\), equivalently \(D\mid d\).  For every \(i\), the
point \(s+q_i d=s+a_i d/D\) is an integer, and hence \(D\mid a_i d\).  Since
\(D\) is the least common denominator of \(q_0,\dots,q_k\), we have
\(\gcd(a_1,\dots,a_{k-1},D)=1\).  Therefore there exist integers
\(c_1,\dots,c_{k-1},m\) such that
$c_1a_1+\cdots+c_{k-1}a_{k-1}+mD=1$.
Multiplying by \(d\), we obtain
$d=c_1a_1d+\cdots+c_{k-1}a_{k-1}d+mDd,$
and every term on the right-hand side is divisible by \(D\).  Thus
\(D\mid d\), so \(p\equiv q\pmod D\).  Moreover,
\[
 x_i(p,q)
 =
 \frac{(D-a_i)s+a_i(s+d)}{D}
 =
 s+\frac{a_i}{D}d
 =
 s+q_i d
\]
for every \(i\).  Hence \(X\) has the form in \textnormal{(ii)}.

Conversely, suppose that \(p,q\in[n]\), \(p\ne q\), and
\(p\equiv q\pmod D\).  Put \(s=p\) and \(d=q-p\).  Then \(s,d\in\Z\),
\(d\ne0\), and \(D\mid d\).  For every \(i\),
\[
 x_i(p,q)
 =
 \frac{(D-a_i)p+a_iq}{D}
 =
 p+\frac{a_i}{D}(q-p)
 =
 s+q_i d.
\]
Thus the set in \eqref{eq-xpq} has the form in \textnormal{(i)}.  Since
\(x_i(p,q)=(1-a_i/D)p+(a_i/D)q\) is a convex combination of \(p\) and \(q\),
and since \(p,q\in[n]\), all its points lie in \([n]\).  This proves the
equivalence.
\end{proof}

\subsection{Ordered constellations}

An \emph{ordered \(Q\)-constellation in \([n]\)} is an ordered tuple
\((x_0(p,q),x_1(p,q),\dots,x_k(p,q))\), where \(x_i(p,q)\) is defined by
\eqref{eq-xpq}, \(p,q\in[n]\), \(p\ne q\), and \(p\equiv q\pmod D\).  Its
underlying set is the corresponding \(Q\)-constellation in \([n]\).  For a
coloring \(\chi:[n]\to\{\pm1\}\), let \(\Mord_Q(\chi)\) denote the number of
monochromatic ordered \(Q\)-constellations in \([n]\).

\begin{lemma}\label{lem:distinct}
Every ordered \(Q\)-constellation in \([n]\) consists of \(k+1\) distinct
integers.
\end{lemma}

\begin{proof}
Consider the affine map \(\phi_{p,q}(t)=p+t(q-p)\).  Since \(p\ne q\), this
map is injective.  Moreover, \(x_i(p,q)=\phi_{p,q}(a_i/D)\) for every \(i\).
Since \(a_0,\dots,a_k\) are distinct, the points
\(x_0(p,q),\dots,x_k(p,q)\) are distinct.
\end{proof}

\begin{lemma}\label{lem:actual-ordered}
Suppose that \(p\ne q\), \(p'\ne q'\), and
\[
\{x_i(p,q):0\le i\le k\}=\{x_i(p',q'):0\le i\le k\}.
\]
Then \(\{p,q\}=\{p',q'\}\).  Moreover:
\begin{enumerate}[label=\textnormal{(\roman*)}]
\item if \((p',q')=(p,q)\), then the two pairs determine the same ordered
\(Q\)-constellation in \([n]\);
\item if \((p',q')=(q,p)\), then the two pairs determine the same
\(Q\)-constellation in \([n]\) if and only if \(Q\) is symmetric.
\end{enumerate}
Consequently, every \(Q\)-constellation in \([n]\) corresponds to exactly one
ordered \(Q\)-constellation in \([n]\) when \(Q\) is nonsymmetric, and to
exactly two ordered \(Q\)-constellations in \([n]\) when \(Q\) is symmetric.
\end{lemma}

\begin{proof}
Recall that \(a_0=0\) and \(a_k=D\).  From \eqref{eq-xpq},
\(x_0(p,q)=p\) and \(x_k(p,q)=q\).  Since \(0\le a_i\le D\), every
\(x_i(p,q)\) lies in the closed interval with endpoints \(p\) and \(q\).  Thus
\(\{p,q\}\) is precisely the set of the two extreme points of
\(\{x_i(p,q):0\le i\le k\}\).  The same is true for \(\{p',q'\}\).  Therefore
\(\{p,q\}=\{p',q'\}\), and hence either \((p',q')=(p,q)\) or
\((p',q')=(q,p)\).  Part \textnormal{(i)} is immediate.

For part \textnormal{(ii)}, assume \((p',q')=(q,p)\).  Since \(p\ne q\), the
affine map \(\phi_{p,q}(t)=p+t(q-p)\) is injective, and
\[
x_i(p,q)=\phi_{p,q}(a_i/D),
\qquad
x_i(q,p)=\phi_{p,q}(1-a_i/D).
\]
Hence
\[
\{x_i(p,q):0\le i\le k\}=\{x_i(q,p):0\le i\le k\}
\]
if and only if
\(\{a_i/D:0\le i\le k\}=\{1-a_i/D:0\le i\le k\}\), equivalently
\(\{a_0,\dots,a_k\}=\{D-a_0,\dots,D-a_k\}\).  Since
\(0=a_0<a_1<\cdots<a_k=D\), this happens if and only if
\(a_i=D-a_{k-i}\) for all \(i\), equivalently \(q_i+q_{k-i}=1\) for all
\(i\).  This is precisely the symmetry of \(Q\).

Consequently, every \(Q\)-constellation in \([n]\) arises from at most two ordered endpoint pairs, namely \((p,q)\) and \((q,p)\). By part \textnormal{(ii)}, both pairs determine the same \(Q\)-constellation in \([n]\) if and only if \(Q\) is symmetric. Hence every \(Q\)-constellation in \([n]\) corresponds to exactly one ordered \(Q\)-constellation when \(Q\) is nonsymmetric, and to exactly two ordered \(Q\)-constellations in \([n]\) when \(Q\) is symmetric.
\end{proof}

\subsection{Total counts and the random coefficient}\label{subsec:total-counts}

We now justify the random-coloring coefficient used in the introduction.  For
\(r\in\{0,1,\dots,D-1\}\), let
\(A_r=\{m\in[n]:m\equiv r\pmod D\}\) and \(N_r=|A_r|\).  Then
\(N_r=n/D+O_D(1)\), uniformly in \(r\).  Let \(N_Q^{\ord}(n)\) and \(N_Q(n)\)
denote, respectively, the total numbers of ordered \(Q\)-constellations in
\([n]\) and \(Q\)-constellations in \([n]\).

Ordered \(Q\)-constellations in \([n]\) are exactly ordered pairs of distinct
endpoints from a common residue class modulo \(D\).  Hence
\[
N_Q^{\ord}(n)=\sum_{r=0}^{D-1}N_r(N_r-1)=\frac{n^2}{D}+O_D(n).
\]
By Lemma~\ref{lem:actual-ordered}, it follows that
\begin{equation}\label{eq:Nact}
N_Q(n)=
\begin{cases}
\dfrac{n^2}{D}+O_D(n), & Q \text{ is nonsymmetric},\\[8pt]
\dfrac{n^2}{2D}+O_D(n), & Q \text{ is symmetric}.
\end{cases}
\end{equation}

Recall that \(\chi_{\rand}\) is a uniformly random two-coloring of \([n]\), in
which the colors of the points \(1,\dots,n\) are independent and uniformly
distributed on \(\{\pm1\}\).  By Lemma~\ref{lem:distinct}, every
\(Q\)-constellation in \([n]\) has \(k+1\) distinct points.  Therefore a fixed
\(Q\)-constellation in \([n]\) is monochromatic under \(\chi_{\rand}\) with
probability \(2\cdot 2^{-(k+1)}=2^{-k}\).  Thus, $\E\!\left[M_Q(\chi_{\rand})\right]=2^{-k}N_Q(n)$.
Combining this with \eqref{eq:Nact} gives the random-coloring coefficient
\begin{equation}\label{eq:Nact2}
\gamma_{\rand}(Q)=
\begin{cases}
\dfrac{1}{2^kD}, & Q \text{ is nonsymmetric},\\[8pt]
\dfrac{1}{2^{k+1}D}, & Q \text{ is symmetric}.
\end{cases}
\end{equation}

\section{A continuous model and the discrete-to-continuous transfer}
\label{sec:continuous-transfer}

The purpose of this section is to pass from the discrete counting problem to a
continuous bias model on \([0,1]\).  The endpoint parametrization from
Section~\ref{sec:preliminaries} suggests normalizing the endpoints
\(p,q\in[n]\) by \(x=p/n\) and \(y=q/n\).  This leads to a continuous
functional whose value describes the limiting monochromatic probability for a
biased product coloring.

For normalized endpoints \(x=p/n\) and \(y=q/n\), the normalized coordinate of
the \(i\)-th point of the corresponding \(Q\)-constellation is
\[
\frac{x_i(p,q)}{n}
=
\left(1-\frac{a_i}{D}\right)x+\frac{a_i}{D}y
=
(1-q_i)x+q_i y.
\]
For \(x,y\in[0,1]\), define the affine forms
\[
L_i(x,y)
=
(1-q_i)x+q_i y
=
\left(1-\frac{a_i}{D}\right)x+\frac{a_i}{D}y,
\qquad 0\le i\le k.
\]
Since \(0\le q_i\le1\), each \(L_i(x,y)\) lies in \([0,1]\).  Thus
\(L_i(x,y)\) is the normalized position of the \(i\)-th point of the
constellation with normalized endpoints \(x\) and \(y\).

For any measurable \(b:[0,1]\to[-1,1]\), define
\begin{equation}\label{eq:Phi}
\Phi_Q(b)
=
\int_0^1\!\int_0^1
\left(
\prod_{i=0}^k\frac{1+b(L_i(x,y))}{2}
+
\prod_{i=0}^k\frac{1-b(L_i(x,y))}{2}
\right)
\,dy\,dx.
\end{equation}
Here \(b\) is a local bias: at a point \(z\in[0,1]\), the probabilities of
the colors \(+1\) and \(-1\) are \((1+b(z))/2\) and \((1-b(z))/2\),
respectively.  Thus, for fixed normalized endpoints \((x,y)\), the two
products in \eqref{eq:Phi} are the probabilities that all points of the
corresponding constellation are colored \(+1\) and \(-1\), respectively.
Therefore \(\Phi_Q(b)\) is the average monochromatic probability in the
continuous endpoint model.  In particular, if \(b\equiv0\), then
\(\Phi_Q(0)=2^{-k}\).

Our continuous goal is to find a bias \(b\) for which \(\Phi_Q(b)<2^{-k}\).
The transfer result below will show that, when \(b\) is Lipschitz, such an
inequality gives a product coloring of \([n]\) whose expected number of
monochromatic \(Q\)-constellations is below the random-coloring main term.

For each \(n\), let \(\chi_{n,b}\) be the product coloring of \([n]\) defined
by
\[
\Prob\bigl(\chi_{n,b}(m)=1\bigr)
=
\frac{1+b(m/n)}{2},
\qquad
\Prob\bigl(\chi_{n,b}(m)=-1\bigr)
=
\frac{1-b(m/n)}{2},
\]
independently for \(1\le m\le n\), where \(\Prob\) denotes probability.  In
particular, \(\chi_{n,0}\) has the same distribution as \(\chi_{\rand}\).

\subsection{Uniform residue-class Riemann sums}
Having obtained the continuous average monochromatic probability $\Phi_Q(b)$, we now aim to connect it to the expected number of monochromatic $Q$-constellations in the discrete setting. A critical preliminary observation is that the endpoints $p,q$ of a discrete $Q$-constellation are not arbitrary integer pairs, but must satisfy the congruence $p \equiv q \pmod{D}$. For this reason, standard Riemann sum approximations are not directly applicable, and we instead need a dedicated result for Riemann sums sampled over residue classes modulo $D$.

For \(r\in\{0,1,\dots,D-1\}\), let
\[
Y_r
=
\left\{
\frac{r+D\beta}{n}: 1\le r+D\beta\le n,\ \beta\in\Z
\right\}.
\]
Thus \(Y_r\) is the normalized set of integers in \([n]\) congruent to \(r\)
modulo \(D\).

\begin{lemma}\label{lem:Riemann}
Let \(f:[0,1]\to\R\) be Lipschitz.  For each
\(r\in\{0,1,\dots,D-1\}\), define \(T_r(f)=\sum_{y\in Y_r}f(y)\).  Then,
uniformly in \(r\),
\begin{equation}\label{eq:Riemann-main}
T_r(f)=\frac{n}{D}\int_0^1 f(t)\,dt+O_f(1).
\end{equation}
More precisely, if \(L_f\) is a Lipschitz constant for \(f\) and
\(M_f=\|f\|_\infty\), then
\begin{equation}\label{eq:Riemann-precise}
\left|\int_0^1 f(t)\,dt-\frac{D}{n}T_r(f)\right|
\le (L_f+M_f)\frac{D}{n}
\end{equation}
for every \(r\).
\end{lemma}

\begin{proof}
Set \(h=D/n\), and let \(N=|Y_r|\).  If \(N\ge1\), write the elements of
\(Y_r\) in increasing order as \(y_0<y_1<\cdots<y_{N-1}\).

We first treat the cases \(N=0\) and \(N=1\).  If \(N=0\), then
\(T_r(f)=0\).  This can occur only when \(n<D\), hence \(h>1\).  Therefore
\[
\left|\int_0^1 f(t)\,dt-hT_r(f)\right|
=
\left|\int_0^1 f(t)\,dt\right|
\le M_f
\le (L_f+M_f)h.
\]
If \(N=1\), say \(Y_r=\{y_0\}\), then
\[
\begin{aligned}
\left|\int_0^1 f(t)\,dt-hf(y_0)\right|
&\le
\int_0^1 |f(t)-f(y_0)|\,dt
+
|1-h|\,|f(y_0)|  \\
&\le
L_f\int_0^1 |t-y_0|\,dt
+
M_f|1-h|.
\end{aligned}
\]
Since \(N=1\), necessarily \(n<2D\), and hence \(h=D/n>1/2\).  Also
\(\int_0^1 |t-y_0|\,dt\le 1/2\le h\) and \(|1-h|\le h\).  Hence
\[
\left|\int_0^1 f(t)\,dt-hT_r(f)\right|
\le (L_f+M_f)h.
\]

It remains to consider the case \(N\ge2\).  Since consecutive integers in a
fixed residue class modulo \(D\) differ by \(D\), consecutive sampling points
satisfy
\begin{equation}\label{eq:grid-spacing}
y_{j+1}-y_j=h
\qquad (0\le j\le N-2).
\end{equation}
We partition \([0,1]\) by Voronoi cells centered at the sampling points.
Define
\[
m_0=0,
\qquad
m_j=\frac{y_{j-1}+y_j}{2}\quad (1\le j\le N-1),
\qquad
m_N=1,
\]
and set \(I_j=[m_j,m_{j+1}]\) for \(0\le j\le N-1\).  Then
\([0,1]=\bigcup_{j=0}^{N-1}I_j\).

For every \(t\in I_j\), we have \(|t-y_j|\le h\).  Indeed, for
\(1\le j\le N-2\), \eqref{eq:grid-spacing} gives
\(I_j=[y_j-h/2,y_j+h/2]\).  For the left boundary cell,
\[
I_0=[0,(y_0+y_1)/2]\subseteq [y_0-h,y_0+h],
\]
because \(0<y_0\le h\) and \(y_1=y_0+h\).  The right boundary cell is
analogous, using \(0\le 1-y_{N-1}<h\).

Moreover, every interior cell has length \(h\), and only the two boundary
cells can have length different from \(h\).  Since
\[
|I_0|=y_0+\frac h2,
\qquad
|I_{N-1}|=1-y_{N-1}+\frac h2,
\]
we have
\[
\bigl||I_0|-h\bigr|
\le \frac h2,
\qquad
\bigl||I_{N-1}|-h\bigr|
\le \frac h2.
\]
Thus
\begin{equation}\label{eq:length-deviation}
\sum_{j=0}^{N-1}\bigl||I_j|-h\bigr|\le h.
\end{equation}

By Lipschitz continuity,
\[
\left|\int_{I_j}f(t)\,dt-|I_j|f(y_j)\right|
\le
\int_{I_j}|f(t)-f(y_j)|\,dt
\le
L_fh|I_j|.
\]
Summing over \(j\), we get
\begin{equation}\label{eq:int-vs-weighted}
\left|
\int_0^1 f(t)\,dt
-
\sum_{j=0}^{N-1}|I_j|f(y_j)
\right|
\le L_fh.
\end{equation}
Next, by \eqref{eq:length-deviation} and \(|f(y_j)|\le M_f\),
\begin{equation}\label{eq:weighted-vs-riemann}
\left|
\sum_{j=0}^{N-1}|I_j|f(y_j)
-
h\sum_{j=0}^{N-1}f(y_j)
\right|
\le M_fh.
\end{equation}
Combining \eqref{eq:int-vs-weighted} and \eqref{eq:weighted-vs-riemann}, we
obtain
\[
\left|\int_0^1 f(t)\,dt-hT_r(f)\right|
\le (L_f+M_f)h.
\]
Since \(h=D/n\), this is \eqref{eq:Riemann-precise}.  Multiplying by
\(h^{-1}=n/D\) gives \eqref{eq:Riemann-main}.
\end{proof}

The uniformity in the residue class allows us to apply the estimate twice,
first in one endpoint variable and then in the other, without losing control of
the error term.

\subsection{Two-dimensional discrete sums versus the continuous functional}

Recall that \(\Mord_Q(\chi)\) denotes the number of monochromatic ordered
\(Q\)-constellations in \([n]\) under a coloring \(\chi\), while
\(M_Q(\chi)\) denotes the number of monochromatic \(Q\)-constellations in
\([n]\).  For the product coloring \(\chi_{n,b}\), define
\[
\Mord_Q(n;b)=\E\!\left[\Mord_Q(\chi_{n,b})\right],
\qquad
M_Q(n;b)=\E\!\left[M_Q(\chi_{n,b})\right].
\]

We can now establish a precise correspondence between the expected number of monochromatic $Q$-constellations in the discrete setting and the continuous functional $\Phi_Q(b)$, which forms the core result of the following proposition.
\begin{proposition}\label{prop:transfer}
If \(b:[0,1]\to[-1,1]\) is Lipschitz, then
\begin{equation}\label{eq:transfer-ord}
\Mord_Q(n;b)=\frac{n^2}{D}\,\Phi_Q(b)+O_{Q,b}(n).
\end{equation}
If \(Q\) is nonsymmetric, then
\begin{equation}\label{eq:transfer-act-nonsym}
M_Q(n;b)=\Mord_Q(n;b)=\frac{n^2}{D}\,\Phi_Q(b)+O_{Q,b}(n).
\end{equation}
If \(Q\) is symmetric, then
\begin{equation}\label{eq:transfer-act-sym}
M_Q(n;b)=\frac12\Mord_Q(n;b)=\frac{n^2}{2D}\,\Phi_Q(b)+O_{Q,b}(n).
\end{equation}
\end{proposition}

\begin{proof}
Define
\[
H_b(x,y)=
\prod_{i=0}^k\frac{1+b(L_i(x,y))}{2}
+
\prod_{i=0}^k\frac{1-b(L_i(x,y))}{2}.
\]
Fix endpoints \(p\ne q\) with \(p\equiv q\pmod D\).  By
Lemma~\ref{lem:distinct}, the points \(x_0(p,q),\dots,x_k(p,q)\) are
distinct, so their colors are independent.  The probability that this ordered
copy is monochromatic is therefore
$H_b\!\left(\frac pn,\frac qn\right)$,
because
\[
\frac{x_i(p,q)}{n}
=
\left(1-\frac{a_i}{D}\right)\frac pn+\frac{a_i}{D}\frac qn
=
L_i\!\left(\frac pn,\frac qn\right).
\]
Hence
\[
\Mord_Q(n;b)=
\sum_{\substack{1\le p,q\le n\\ p\ne q,\ p\equiv q\pmod D}}
H_b\!\left(\frac pn,\frac qn\right).
\]
Since \(0\le H_b(x,y)\le1\), adding the diagonal \(p=q\) changes the sum by
\(O(n)\).  Thus
\[
\Mord_Q(n;b)=
\sum_{\substack{1\le p,q\le n\\ p\equiv q\pmod D}}
H_b\!\left(\frac pn,\frac qn\right)+O(n).
\]
For each residue class \(r\in\{0,\dots,D-1\}\), define
\[
S_r=
\sum_{\substack{1\le p,q\le n\\ p\equiv q\equiv r\pmod D}}
H_b\!\left(\frac pn,\frac qn\right).
\]
Then
\begin{equation}\label{eq:Sr-sum}
\Mord_Q(n;b)=\sum_{r=0}^{D-1}S_r+O(n).
\end{equation}

We now compare each \(S_r\) with the corresponding integral.  Since \(b\) is
Lipschitz and the maps \(L_i\) are affine, each composition \(b\circ L_i\) is
Lipschitz on \([0,1]^2\).  Because these functions are bounded by \(1\),
finite sums and products of them remain Lipschitz.  Hence \(H_b\) is
Lipschitz on \([0,1]^2\); let \(C_{Q,b}\) be a Lipschitz constant for
\(H_b\).

Fix \(y\in[0,1]\), and set \(f_y(x)=H_b(x,y)\).  The Lipschitz constant and
supremum norm of \(f_y\) are bounded independently of \(y\).  Applying
Lemma~\ref{lem:Riemann} to \(f_y\), uniformly in \(y\), gives for every
residue class \(r\)
\begin{equation}\label{eq:first-Riemann}
\sum_{\substack{1\le p\le n\\ p\equiv r\pmod D}}
H_b\!\left(\frac pn,y\right)
=
\frac{n}{D}\int_0^1H_b(x,y)\,dx+O_{Q,b}(1).
\end{equation}
Summing \eqref{eq:first-Riemann} over \(y\in Y_r\), we obtain
\[
S_r=
\sum_{y\in Y_r}\sum_{\substack{1\le p\le n\\ p\equiv r\ ({\rm mod}\ D)}} H_b\!\left(\frac pn,y\right)
=
\frac{n}{D}\sum_{y\in Y_r}\int_0^1 H_b(x,y)\,dx+O_{Q,b}(|Y_r|).
\]
Since \(|Y_r|=n/D+O_D(1)\), this becomes
\begin{equation}\label{eq:Sr-G}
S_r
=
\frac{n}{D}\sum_{y\in Y_r}G(y)+O_{Q,b}(n),
\qquad
G(y)=\int_0^1H_b(x,y)\,dx.
\end{equation}
The function \(G\) is Lipschitz.  Indeed, for \(y,y'\in[0,1]\),
\[
\begin{aligned}
|G(y)-G(y')|
&=\left|\int_0^1\bigl(H_b(x,y)-H_b(x,y')\bigr)\,dx\right|\\
&\le \int_0^1|H_b(x,y)-H_b(x,y')|\,dx\\
&\le \int_0^1 C_{Q,b}|y-y'|\,dx\\
&=C_{Q,b}|y-y'|.
\end{aligned}
\]
Applying Lemma~\ref{lem:Riemann} once more, now to \(G\), yields
\begin{equation}\label{eq:second-Riemann}
\sum_{y\in Y_r}G(y)
=
\frac{n}{D}\int_0^1G(y)\,dy+O_{Q,b}(1).
\end{equation}
Substituting \eqref{eq:second-Riemann} into \eqref{eq:Sr-G} gives
\[
S_r
=
\frac{n^2}{D^2}\int_0^1G(y)\,dy+O_{Q,b}(n)
=
\frac{n^2}{D^2}\Phi_Q(b)+O_{Q,b}(n).
\]
Summing over \(r=0,\dots,D-1\) in \eqref{eq:Sr-sum} gives
\eqref{eq:transfer-ord}.

Finally, Lemma~\ref{lem:actual-ordered} shows that ordered
\(Q\)-constellations are in bijection with underlying \(Q\)-constellations in
the nonsymmetric case, while the map from ordered constellations to underlying
constellations is two-to-one in the symmetric case.  This gives
\eqref{eq:transfer-act-nonsym} and \eqref{eq:transfer-act-sym}.
\end{proof}

Thus Proposition~\ref{prop:transfer} reduces the problem to constructing a
Lipschitz bias \(b:[0,1]\to[-1,1]\) with \(\Phi_Q(b)<2^{-k}\).  Any such fixed
\(b\) gives, for all sufficiently large \(n\), a product coloring whose
expected monochromatic density is strictly below the random-coloring
coefficient \(\gamma_{\rand}(Q)\).

\subsection{Second variation at the random point}

We next analyze the functional \(\Phi_Q\) near the unbiased point
\(b\equiv0\), which corresponds to the uniformly random coloring.  The aim is
to identify a perturbation direction along which \(\Phi_Q\) decreases.  Fix a
measurable function \(g:[0,1]\to[-1,1]\), let \(\lambda\in\R\) with
\(|\lambda|\le1\), and set \(b_\lambda=\lambda g\).  Then \(b_\lambda\) is an
admissible bias function.  We expand \(\Phi_Q(b_\lambda)\) as
\(\lambda\to0\); the first nonconstant term is a quadratic form in \(g\).
Thus a function \(g\) for which this quadratic form is negative will give
\(\Phi_Q(\lambda g)<2^{-k}\) for all sufficiently small positive \(\lambda\).

Define the quadratic form
\begin{equation}\label{eq:T2}
T_2(g)=
\sum_{0\le i<j\le k}
\int_0^1\!\int_0^1 g(L_i(x,y))g(L_j(x,y))\,dy\,dx.
\end{equation}

\begin{lemma}\label{lem:even-expansion}
For every measurable \(g:[0,1]\to[-1,1]\),
\begin{equation}\label{eq:Phi-expansion}
\Phi_Q(\lambda g)=2^{-k}+2^{-k}\lambda^2T_2(g)+O_k(\lambda^4),
\qquad |\lambda|\le 1,
\end{equation}
where the implicit constant depends only on \(k\).
\end{lemma}

\begin{proof}
Fix \((x,y)\), and write \(g_i=g(L_i(x,y))\) for \(0\le i\le k\).  The
integrand in \eqref{eq:Phi} is
\[
2^{-(k+1)}
\left(
\prod_{i=0}^k(1+\lambda g_i)
+
\prod_{i=0}^k(1-\lambda g_i)
\right).
\]
Expanding the two products gives
\[
\prod_{i=0}^k(1+\lambda g_i)
=
\sum_{S\subseteq\{0,\dots,k\}}
\lambda^{|S|}\prod_{i\in S}g_i
\]
and
\[
\prod_{i=0}^k(1-\lambda g_i)
=
\sum_{S\subseteq\{0,\dots,k\}}
(-1)^{|S|}\lambda^{|S|}\prod_{i\in S}g_i.
\]
Adding the two expansions cancels every odd-degree term and leaves
\[
\prod_{i=0}^k(1+\lambda g_i)
+
\prod_{i=0}^k(1-\lambda g_i)
=
2\sum_{\substack{S\subseteq\{0,\dots,k\}\\ |S|\ \mathrm{even}}}
\lambda^{|S|}\prod_{i\in S}g_i.
\]
The terms of degrees \(0\) and \(2\) are
\[
2+2\lambda^2\sum_{0\le i<j\le k}g_i g_j.
\]
After integration over \([0,1]^2\), these contribute \(2^{-k}\) and
\(2^{-k}\lambda^2T_2(g)\), respectively.  The remaining terms have even degree
at least \(4\).  Since \(|g_i|\le1\), their total contribution is bounded in
absolute value by
\[
2^{-k}
\sum_{r\ge2}\binom{k+1}{2r}|\lambda|^{2r}
\le
2^{-k}\left(\sum_{r\ge2}\binom{k+1}{2r}\right)|\lambda|^4,
\]
for \(|\lambda|\le1\).  This is \(O_k(\lambda^4)\), and the lemma follows.
\end{proof}

Consequently, if \(T_2(g)<0\), then \(\Phi_Q(\lambda g)<2^{-k}\) for all
sufficiently small \(\lambda>0\).  Thus it remains to construct an admissible
direction \(g\) with negative quadratic variation.

\section{An explicit negative direction}
\label{sec:negative-direction}

We now construct a smooth admissible direction \(g\) with \(T_2(g)<0\).  Set
\(u(x)=\cos(2\pi D x)\) and \(v(x)=\cos(2\pi(D+1)x)\).  The mode \(u\) is a
zero direction for \(T_2\), while the mixed bilinear term between \(u\) and the
adjacent mode \(v\) has a fixed positive sign.  A small perturbation of \(u\)
in the direction \(-v\) will therefore give the desired negative direction.

We first show that the \(D\)-frequency mode is a zero direction.

\begin{lemma}\label{lem:T2u-zero}
We have \(T_2(u)=0\).
\end{lemma}

\begin{proof}
For \(0\le i<j\le k\), set
\[
I_{i,j}
=
\int_0^1\!\int_0^1
u(L_i(x,y))u(L_j(x,y))\,dy\,dx.
\]
It suffices to show that \(I_{i,j}=0\) for every \(i<j\).  Since
\(q_i=a_i/D\),
\[
D L_i(x,y)=(D-a_i)x+a_i y,
\]
and therefore
\[
u(L_i(x,y))=\cos\bigl(2\pi((D-a_i)x+a_i y)\bigr),
\qquad
u(L_j(x,y))=\cos\bigl(2\pi((D-a_j)x+a_j y)\bigr).
\]
Let
\[
A=2\pi((D-a_i)x+a_i y),
\qquad
B=2\pi((D-a_j)x+a_j y).
\]
Using \(\cos A\cos B=\frac12\cos(A-B)+\frac12\cos(A+B)\), we get
\[
I_{i,j}
=
\frac12\int_0^1\!\int_0^1\cos(A-B)\,dy\,dx
+
\frac12\int_0^1\!\int_0^1\cos(A+B)\,dy\,dx.
\]
For the first integral,
$A-B=2\pi\bigl((a_j-a_i)x+(a_i-a_j)y\bigr)$,
and the frequency vector \((a_j-a_i,a_i-a_j)\) is nonzero because
\(a_i\ne a_j\).  For the second integral,
$A+B=2\pi\bigl((2D-a_i-a_j)x+(a_i+a_j)y\bigr)$,
and the frequency vector \((2D-a_i-a_j,a_i+a_j)\) is nonzero because
\(j\ge1\) and hence \(a_i+a_j\ge a_j>0\).

In each case, writing the cosine in exponential form reduces the integral to
a linear combination of terms
\[
\int_0^1\!\int_0^1 e^{2\pi \mathrm i(mx+ny)}\,dy\,dx
=
\left(\int_0^1e^{2\pi \mathrm i m x}\,dx\right)
\left(\int_0^1e^{2\pi \mathrm i n y}\,dy\right),
\]
where \(m,n\in\Z\) and \((m,n)\ne(0,0)\).  At least one of the two
one-dimensional integrals vanishes.  Thus both integrals above are zero, and
therefore \(I_{i,j}=0\).  Summing over all \(0\le i<j\le k\), we obtain
\(T_2(u)=0\).
\end{proof}

We next calculate the mixed correlations of \(u\) and \(v\).  Set
\begin{equation}\label{eq:J-def}
J(\theta)=
\int_0^1 e^{2\pi \mathrm i \theta x}\,dx
=
\begin{cases}
\dfrac{e^{2\pi \mathrm i\theta}-1}{2\pi \mathrm i\theta}, & \theta\ne0,\\[8pt]
1, & \theta=0.
\end{cases}
\end{equation}

\begin{lemma}\label{lem:J-product}
If \(N\in\Z\), then for every real \(\theta\),
\begin{equation}\label{eq:J-product}
J(\theta)J(N-\theta)
=
-\frac{\sin^2(\pi\theta)}{\pi^2\theta(N-\theta)},
\end{equation}
where the right-hand side is understood by continuous extension at
\(\theta=0\) and \(\theta=N\).
\end{lemma}

\begin{proof}
Assume first that \(\theta\ne0,N\).  By \eqref{eq:J-def},
\[
J(\theta)J(N-\theta)
=
\frac{(e^{2\pi \mathrm i\theta}-1)(e^{2\pi \mathrm i(N-\theta)}-1)}
{(2\pi \mathrm i\theta)(2\pi \mathrm i(N-\theta))}.
\]
Since \(N\in\Z\), we have \(e^{2\pi \mathrm iN}=1\), and hence
\(e^{2\pi \mathrm i(N-\theta)}=e^{-2\pi \mathrm i\theta}\).  Therefore the
numerator equals
\[
(e^{2\pi \mathrm i\theta}-1)(e^{-2\pi \mathrm i\theta}-1)
=
2-\bigl(e^{2\pi \mathrm i\theta}+e^{-2\pi \mathrm i\theta}\bigr)
=
4\sin^2(\pi\theta).
\]
Since \((2\pi \mathrm i)^2=-4\pi^2\), the denominator is
\(-4\pi^2\theta(N-\theta)\), and \eqref{eq:J-product} follows.  The endpoint
cases are obtained by taking limits.
\end{proof}

\subsection{Mixed correlations}

Let \(\Re(z)\) denote the real part of the complex number \(z\), that is, if
\(z=\alpha+\beta\mathrm i\) with \(\alpha,\beta\in\R\), then
\(\Re(z)=\alpha\).  For integers \(0\le s<t\le D\), write
\[
L_s(x,y)=\left(1-\frac{s}{D}\right)x+\frac{s}{D}y.
\]
This agrees with the previous notation when \(s=a_i\).  Define
\[
U_{s,t}=
\int_0^1\!\int_0^1 u(L_s(x,y))v(L_t(x,y))\,dy\,dx.
\]

\begin{lemma}\label{lem:Uab}
For every \(0\le s<t\le D\),
\begin{align}\label{eq:Uab}
U_{s,t}
&=
\frac{\sin^2(\pi t/D)}{2\pi^2}
\left[
\frac{1}{(t-s+t/D)(t-s-1+t/D)}
\right. \notag\\
&\hspace{3.6cm}\left.
-
\frac{1}{(s+t+t/D)(2D-s-t+1-t/D)}
\right].
\end{align}
In particular, \(U_{s,t}>0\) when \(t<D\), while \(U_{s,D}=0\).
\end{lemma}

\begin{proof}
Set
\[
A=2\pi D L_s(x,y),
\qquad
B=2\pi(D+1)L_t(x,y).
\]
Then
\[
u(L_s)=\frac{e^{\mathrm iA}+e^{-\mathrm iA}}2,
\qquad
v(L_t)=\frac{e^{\mathrm iB}+e^{-\mathrm iB}}2.
\]
Hence
\[
U_{s,t}
=
\frac12\Re\left(
\int_0^1\!\int_0^1 e^{\mathrm i(A+B)}\,dy\,dx
+
\int_0^1\!\int_0^1 e^{\mathrm i(A-B)}\,dy\,dx
\right).
\]
Write
\[
A+B=2\pi(\beta_+x+\alpha_+y),
\qquad
A-B=2\pi(\beta_-x+\alpha_-y),
\]
where
\[
\alpha_+=s+t+\frac tD,
\qquad
\beta_+=2D-s-t+1-\frac tD,
\]
and
\[
\alpha_-=s-t-\frac tD,
\qquad
\beta_-=t-s-1+\frac tD.
\]
Then the two integrals are \(J(\alpha_+)J(\beta_+)\) and
\(J(\alpha_-)J(\beta_-)\), respectively.  Moreover,
\[
\alpha_++\beta_+=2D+1\in\Z,
\qquad
\alpha_-+\beta_-=-1\in\Z.
\]
By Lemma~\ref{lem:J-product},
\[
J(\alpha_+)J(\beta_+)
=
-\frac{\sin^2(\pi\alpha_+)}{\pi^2\alpha_+\beta_+},
\qquad
J(\alpha_-)J(\beta_-)
=
-\frac{\sin^2(\pi\alpha_-)}{\pi^2\alpha_-\beta_-}.
\]
Since \(s,t\in\Z\),
\[
\sin^2(\pi\alpha_+)=\sin^2\left(\pi\frac tD\right),
\qquad
\sin^2(\pi\alpha_-)=\sin^2\left(\pi\frac tD\right).
\]
Also,
\[
-\alpha_-=t-s+\frac tD,
\qquad
\beta_-=t-s-1+\frac tD.
\]
Substituting these identities gives \eqref{eq:Uab}.

If \(t=D\), then the factor \(\sin^2(\pi t/D)\) vanishes, so
\(U_{s,D}=0\).  Suppose now that \(t<D\).  Then this sine factor is strictly
positive.  Set
\[
P_1=t-s+\frac tD,
\quad
Q_1=t-s-1+\frac tD,
\quad
P_2=s+t+\frac tD,
\quad
Q_2=2D-s-t+1-\frac tD.
\]
All four quantities are positive; in particular,
\(Q_1=t-s-1+t/D\ge t/D>0\).  Moreover,
\[
P_2-P_1=2s\ge0,
\qquad
Q_2-Q_1=2\left(D-t+1-\frac tD\right)>0,
\]
because \(t<D\).  Hence \(P_2Q_2>P_1Q_1\), and the bracketed expression in
\eqref{eq:Uab} is positive.  Thus \(U_{s,t}>0\).
\end{proof}

Similarly, define
\[
V_{s,t}=
\int_0^1\!\int_0^1 v(L_s(x,y))u(L_t(x,y))\,dy\,dx.
\]

\begin{lemma}\label{lem:Vab}
For every \(0\le s<t\le D\),
\begin{align}\label{eq:Vab}
V_{s,t}
&=
\frac{\sin^2(\pi s/D)}{2\pi^2}
\left[
\frac{1}{(t-s-s/D)(t-s+1-s/D)}
\right. \notag\\
&\hspace{3.6cm}\left.
-
\frac{1}{(s+t+s/D)(2D-s-t+1-s/D)}
\right].
\end{align}
In particular, \(V_{s,t}>0\) when \(s>0\), while \(V_{0,t}=0\).
\end{lemma}

\begin{proof}
The calculation is parallel to that in Lemma~\ref{lem:Uab}.  Set
\[
A=2\pi(D+1)L_s(x,y),
\qquad
B=2\pi D L_t(x,y).
\]
Then
\[
v(L_s)=\frac{e^{\mathrm iA}+e^{-\mathrm iA}}2,
\qquad
u(L_t)=\frac{e^{\mathrm iB}+e^{-\mathrm iB}}2.
\]
Hence
\[
V_{s,t}
=
\frac12\Re\left(
\int_0^1\!\int_0^1 e^{\mathrm i(A+B)}\,dy\,dx
+
\int_0^1\!\int_0^1 e^{\mathrm i(A-B)}\,dy\,dx
\right).
\]
Write
\[
A+B=2\pi(\beta_+x+\alpha_+y),
\qquad
A-B=2\pi(\beta_-x+\alpha_-y),
\]
where
\[
\alpha_+=s+t+\frac sD,
\qquad
\beta_+=2D-s-t+1-\frac sD,
\]
and
\[
\alpha_-=s-t+\frac sD,
\qquad
\beta_-=t-s+1-\frac sD.
\]
Then the two integrals are \(J(\alpha_+)J(\beta_+)\) and
\(J(\alpha_-)J(\beta_-)\), respectively.  Moreover,
\[
\alpha_++\beta_+=2D+1\in\Z,
\qquad
\alpha_-+\beta_-=1\in\Z.
\]
By Lemma~\ref{lem:J-product},
\[
J(\alpha_+)J(\beta_+)
=
-\frac{\sin^2(\pi\alpha_+)}{\pi^2\alpha_+\beta_+},
\qquad
J(\alpha_-)J(\beta_-)
=
-\frac{\sin^2(\pi\alpha_-)}{\pi^2\alpha_-\beta_-}.
\]
Since \(s\in\Z\),
\[
\sin^2(\pi\alpha_+)=\sin^2\left(\pi\frac sD\right),
\qquad
\sin^2(\pi\alpha_-)=\sin^2\left(\pi\frac sD\right).
\]
Also,
\[
-\alpha_-=t-s-\frac sD,
\qquad
\beta_-=t-s+1-\frac sD.
\]
Substituting these identities gives \eqref{eq:Vab}.

If \(s=0\), then the factor \(\sin^2(\pi s/D)\) vanishes, so
\(V_{0,t}=0\).  Suppose now that \(s>0\).  Then this sine factor is strictly
positive.  Set
\[
P_1=t-s-\frac sD,
\quad
Q_1=t-s+1-\frac sD,
\quad
P_2=s+t+\frac sD,
\quad
Q_2=2D-s-t+1-\frac sD.
\]
All four quantities are positive; in particular, \(P_1\ge 1-s/D>0\).  Moreover,
\[
P_2-P_1=2s+\frac{2s}{D}>0,
\qquad
Q_2-Q_1=2(D-t)\ge0.
\]
Thus \(P_2Q_2>P_1Q_1\), and the bracketed expression in \eqref{eq:Vab} is
positive.  Hence \(V_{s,t}>0\).
\end{proof}

Summing the mixed correlations gives the required positive mixed term.  Define
\begin{equation}\label{eq:Lambda}
\begin{aligned}
\Lambda_Q
=
\sum_{0\le i<j\le k}
&\Biggl(
\int_0^1\!\int_0^1 u(L_i(x,y))v(L_j(x,y))\,dy\,dx \\
&\qquad
+
\int_0^1\!\int_0^1 v(L_i(x,y))u(L_j(x,y))\,dy\,dx
\Biggr).
\end{aligned}
\end{equation}

\begin{corollary}\label{cor:Lambda-positive}
We have \(\Lambda_Q>0\).
\end{corollary}

\begin{proof}
For \(i<j\), the affine forms \(L_i\) and \(L_j\) are the forms \(L_{a_i}\)
and \(L_{a_j}\).  Hence Lemmas~\ref{lem:Uab} and \ref{lem:Vab} give
\[
\int_0^1\!\int_0^1 u(L_i(x,y))v(L_j(x,y))\,dy\,dx
=U_{a_i,a_j}\ge0
\]
and
\[
\int_0^1\!\int_0^1 v(L_i(x,y))u(L_j(x,y))\,dy\,dx
=V_{a_i,a_j}\ge0.
\]
Since \(k\ge2\), there is an interior point, and hence \(0<a_1<D\).  For the
pair \((a_0,a_1)=(0,a_1)\), Lemma~\ref{lem:Uab} gives \(U_{0,a_1}>0\), while
Lemma~\ref{lem:Vab} gives \(V_{0,a_1}=0\).  Thus at least one term in
\eqref{eq:Lambda} is strictly positive, and all terms are nonnegative.
Therefore \(\Lambda_Q>0\).
\end{proof}

Finally, we construct the function $g$ satisfying $T_2(g) < 0$.
We write \(C^\infty([0,1])\) for the restrictions to \([0,1]\) of smooth
functions defined on an open neighborhood of \([0,1]\).

\begin{proposition}\label{prop:negative-direction}
There exists a function \(g\in C^\infty([0,1])\) with
\(g:[0,1]\to[-1,1]\) such that \(T_2(g)<0\).
\end{proposition}

\begin{proof}
Define the symmetric bilinear form
\[
B_2(f,h)
=
\frac12
\sum_{0\le i<j\le k}
\int_0^1\!\int_0^1
\bigl(
f(L_i(x,y))h(L_j(x,y))
+
h(L_i(x,y))f(L_j(x,y))
\bigr)
\,dy\,dx.
\]
Then \(T_2(f)=B_2(f,f)\), and by \eqref{eq:Lambda},
\(2B_2(u,v)=\Lambda_Q>0\).  By Lemma~\ref{lem:T2u-zero}, \(T_2(u)=0\).  Hence,
for every \(\tau>0\),
\[
T_2(u-\tau v)
=
T_2(u)-2\tau B_2(u,v)+\tau^2T_2(v)
=
-\tau\Lambda_Q+\tau^2T_2(v).
\]
Since \(|v|\le1\), we have
\[
|T_2(v)|
\le
\sum_{0\le i<j\le k}
\int_0^1\!\int_0^1
|v(L_i(x,y))v(L_j(x,y))|\,dy\,dx
\le
\binom{k+1}{2}.
\]
Choose
\[
0<\tau<\frac{\Lambda_Q}{2\binom{k+1}{2}}.
\]
Then
\[
T_2(u-\tau v)<-\frac{\tau\Lambda_Q}{2}<0.
\]
Set
\[
g(x)=\frac{u(x)-\tau v(x)}{1+\tau}.
\]
Since \(|u|,|v|\le1\), we have \(|g|\le1\).  Also
\(g\in C^\infty([0,1])\).  Since \(T_2\) is homogeneous of degree two,
\[
T_2(g)=\frac{1}{(1+\tau)^2}T_2(u-\tau v)<0.
\]
\end{proof}

\section{From a negative direction to deterministic colorings}
\label{sec:deterministic-colorings}

We now use the negative direction constructed in Proposition~\ref{prop:negative-direction}
to obtain a biased product coloring whose expected monochromatic count has
leading coefficient below the random-coloring coefficient.  A standard
conditional-expectation argument then gives a deterministic coloring with the
same asymptotic improvement.

Fix a function \(g\) as in Proposition~\ref{prop:negative-direction}, and set $\eta_Q=-T_2(g)>0$.
By Lemma~\ref{lem:even-expansion}, there exists a constant \(C_k>0\),
depending only on \(k\), such that, for all \(|\lambda|\le1\),
\begin{equation}\label{eq:Phi-lambda-g}
\Phi_Q(\lambda g)
=
2^{-k}
-
2^{-k}\eta_Q\lambda^2
+
R(\lambda),
\qquad
|R(\lambda)|\le C_k\lambda^4.
\end{equation}
Choose \(\lambda>0\) so small that
\begin{equation}\label{eq:lambda-choice}
C_k\lambda^2\le 2^{-k-1}\eta_Q.
\end{equation}
For instance, one may take
\[
0<\lambda\le
\min\left\{1,\sqrt{\frac{\eta_Q}{2^{k+1}C_k}}\right\}.
\]
Then \eqref{eq:Phi-lambda-g} and \eqref{eq:lambda-choice} give
\[
\Phi_Q(\lambda g)
\le
2^{-k}-2^{-k}\eta_Q\lambda^2+C_k\lambda^4
\le
2^{-k}-2^{-k-1}\eta_Q\lambda^2.
\]
Set $c_Q=2^{-k-1}\eta_Q\lambda^2>0$.
Thus
\begin{equation}\label{eq:Phi-improvement}
\Phi_Q(\lambda g)\le 2^{-k}-c_Q.
\end{equation}

Take \(b=\lambda g\) in the product coloring \(\chi_{n,b}\).  Since
\(g\in C^\infty([0,1])\), the function \(b\) is Lipschitz; since \(|g|\le1\)
and \(|\lambda|\le1\), it is an admissible bias function.  By
Proposition~\ref{prop:transfer} and \eqref{eq:Phi-improvement},
\begin{equation}\label{eq:expected-ordered}
\Mord_Q(n;\lambda g)
=
\frac{n^2}{D}\Phi_Q(\lambda g)+O_Q(n)
\le
\left(2^{-k}-c_Q\right)\frac{n^2}{D}+O_Q(n).
\end{equation}
Here the error term is \(O_Q(n)\), since \(g\) and \(\lambda\) are fixed once
\(Q\) is fixed.

It remains to pass from the biased product coloring to a deterministic one. We
use the following standard form of the method of conditional expectations.

\begin{lemma}\label{lem:cond-exp}
Let \(X_1,\dots,X_n\) be independent random variables, each taking finitely
many values, and let \(F=F(X_1,\dots,X_n)\) be an integrable real-valued random
variable.  Then there exists a realization \((x_1,\dots,x_n)\) such that
\[
F(x_1,\dots,x_n)\le \E[F].
\]
Moreover, one may choose such a realization by fixing the variables one at a
time so that the conditional expectation never increases.
\end{lemma}

\begin{proof}
Suppose that \(x_1,\dots,x_{m-1}\) have already been chosen.  The conditional
expectation
\[
\E\!\left[
F\mid X_1=x_1,\dots,X_{m-1}=x_{m-1}
\right]
\]
is a convex combination of the conditional expectations obtained by specifying
the value of \(X_m\).  Hence at least one value \(x_m\) satisfies
\[
\E\!\left[
F\mid X_1=x_1,\dots,X_m=x_m
\right]
\le
\E\!\left[
F\mid X_1=x_1,\dots,X_{m-1}=x_{m-1}
\right].
\]
Iterating this choice for \(m=1,\dots,n\), the final conditional expectation
is \(F(x_1,\dots,x_n)\), and it is at most the original expectation \(\E[F]\).
\end{proof}

\begin{proof}[Proof of Theorem~\ref{thm:main}]
Let \(X_m\in\{\pm1\}\) be the color of \(m\) in the biased product coloring
with bias \(b=\lambda g\), and let \(\chi_X(m)=X_m\).  Apply
Lemma~\ref{lem:cond-exp} to
$F=\Mord_Q(\chi_X)$.
Since \(\E[F]=\Mord_Q(n;\lambda g)\), \eqref{eq:expected-ordered} gives a
deterministic coloring \(\chi_n:[n]\to\{\pm1\}\) such that
\begin{equation}\label{eq:deterministic-ordered}
\Mord_Q(\chi_n)
\le
\left(2^{-k}-c_Q\right)\frac{n^2}{D}+O_Q(n).
\end{equation}

If \(Q\) is nonsymmetric, then ordered \(Q\)-constellations are in bijection
with underlying \(Q\)-constellations by Lemma~\ref{lem:actual-ordered}.
Therefore
\[
M_Q(\chi_n)
\le
\left(\frac{1}{2^kD}-\frac{c_Q}{D}\right)n^2+O_Q(n).
\]
If \(Q\) is symmetric, then the map from ordered constellations to underlying
constellations is two-to-one.  Hence
\[
M_Q(\chi_n)
\le
\left(\frac{1}{2^{k+1}D}-\frac{c_Q}{2D}\right)n^2+O_Q(n).
\]
Taking
\[
\delta_Q=\frac{c_Q}{2D}>0
\]
in both cases, and recalling the definition of \(\gamma_{\rand}(Q)\), we obtain
\[
M_Q(\chi_n)
\le
\bigl(\gamma_{\rand}(Q)-\delta_Q\bigr)n^2+O_Q(n).
\]
This proves Theorem~\ref{thm:main}.
\end{proof}

\section{Applications}
\label{sec:applications}

We record two consequences of Theorem~\ref{thm:main}.  First, it gives a
uniform family of uncommon translation-invariant linear systems in the integer
interval model.  Second, it yields a ground-state bound for a deterministic
antiferromagnetic hypergraph spin system whose hyperedges are rational
constellations.

\subsection{A linear-system consequence}
\label{subsec:linear-system}

In graph Ramsey multiplicity theory, a graph is called common if the uniformly
random edge-coloring asymptotically minimizes the number of monochromatic
copies.  The Burr--Rosta conjecture asserted that all graphs are common, but
this was disproved by Thomason for \(K_4\), with further developments by
Jagger, \v{S}\v{t}ov\'\i\v{c}ek, and Thomason \cite{BR80,Th89,JST96}.
Arithmetic analogues ask the same question for linear configurations under
colorings of groups or intervals.  In finite-field and finite-group settings,
commonness phenomena can behave differently; for example, common and Sidorenko
single equations have been studied and classified in several settings
\cite{SW17,FPZ21,Ver23}.  Theorem~\ref{thm:main} concerns the integer interval
model and yields a uniform family of interval-uncommon translation-invariant
systems.

Let
\[
Q=[0=q_0<q_1<\cdots<q_k=1]
\]
be a rational constellation pattern with \(k\ge2\), and let \(D\) be the least
common denominator of \(q_0,\dots,q_k\).  Write \(a_i=Dq_i\).  Associated to
\(Q\) is the translation-invariant linear system
\[
\mathcal L_Q:\qquad
D x_i=(D-a_i)x_0+a_i x_k,
\qquad 1\le i\le k-1.
\]
Thus \(\mathcal L_Q\) is a system of \(k-1\) equations in the variables
\(x_0,\dots,x_k\), and each equation has coefficient sum zero.

For a coloring \(\chi:[n]\to\{\pm1\}\), let
\(M_{\mathcal L_Q}^{\mathrm{ind}}(\chi)\) denote the number of monochromatic
nonconstant indexed solutions \((x_0,\dots,x_k)\in[n]^{k+1}\) of
\(\mathcal L_Q\), where nonconstant means \(x_0\ne x_k\).  Such a solution is
precisely an endpoint representation of a \(Q\)-constellation: setting
\(s=x_0\) and \(d=x_k-x_0\ne0\), the equations give
\[
x_i=s+q_i d
\qquad (0\le i\le k).
\]

By Lemma~\ref{lem:actual-ordered}, the indexed solution count and the
constellation count differ only by the reversal factor
\[
M_{\mathcal L_Q}^{\mathrm{ind}}(\chi)
=
\nu_Q M_Q(\chi),
\qquad
\nu_Q=
\begin{cases}
1, & Q\text{ is nonsymmetric},\\
2, & Q\text{ is symmetric}.
\end{cases}
\]
Accordingly, the random-coloring coefficient for indexed nonconstant solutions
is
\[
\Gamma_{\rand}(\mathcal L_Q)
=
\nu_Q\gamma_{\rand}(Q)
=
\frac{1}{2^kD}.
\]

\begin{corollary}\label{cor:interval-system-uncommon}
For every rational constellation pattern \(Q\) with \(k\ge2\), the associated
system \(\mathcal L_Q\) is uncommon in the integer interval model.  More
precisely, there exist \(\delta_Q'>0\) and colorings
\(\chi_n:[n]\to\{\pm1\}\) such that
\[
M_{\mathcal L_Q}^{\mathrm{ind}}(\chi_n)
\le
\bigl(\Gamma_{\rand}(\mathcal L_Q)-\delta_Q'\bigr)n^2+O_Q(n).
\]
\end{corollary}

\begin{proof}
By Theorem~\ref{thm:main}, there are colorings \(\chi_n\) and a constant
\(\delta_Q>0\) such that
\[
M_Q(\chi_n)
\le
\bigl(\gamma_{\rand}(Q)-\delta_Q\bigr)n^2+O_Q(n).
\]
Multiplying by \(\nu_Q\), we obtain
\[
M_{\mathcal L_Q}^{\mathrm{ind}}(\chi_n)
\le
\bigl(\nu_Q\gamma_{\rand}(Q)-\nu_Q\delta_Q\bigr)n^2+O_Q(n).
\]
Since \(\Gamma_{\rand}(\mathcal L_Q)=\nu_Q\gamma_{\rand}(Q)\), the claim
follows with \(\delta_Q'=\nu_Q\delta_Q\).
\end{proof}

Thus every rational one-dimensional affine pattern with at least three points
gives an interval-uncommon translation-invariant linear system.  This includes
systems encoding longer arithmetic progressions and higher-codimension
one-dimensional rational configurations.

\subsection{A spin-system consequence}
\label{subsec:spin-system}

Spin systems and spin glasses study energy minimization problems on the
hypercube \(\{-1,+1\}^n\); see, for example, Panchenko's monograph on the
Sherrington--Kirkpatrick model \cite{Panchenko13Book}.  At zero temperature,
one is interested in ground-state energy, as in the mixed \(p\)-spin models
studied by Auffinger and Chen \cite{AuffingerChen17}.  Hypergraph coloring
also has a natural spin-system formulation, where the energy penalizes
monochromatic hyperedges; see, for instance, the work of Bapst, Coja-Oghlan,
and Ra{\ss}mann on random hypergraph \(2\)-coloring
\cite{BapstCojaOghlanRassmann16}.  In contrast with these random models, the
hypergraph below is deterministic and arithmetic: its edges are the
\(Q\)-constellations in \([n]\).

Let \(\mathcal C_Q(n)\) be the family of \(Q\)-constellations in \([n]\), and
view $\mathcal H_Q(n)=([n],\mathcal C_Q(n))$
as a \((k+1)\)-uniform hypergraph.  A spin configuration is a vector
\(\sigma\in\{-1,+1\}^{[n]}\).  For \(e\in\mathcal C_Q(n)\), set
$\delta_e(\sigma)=\mathbf 1\{\sigma \text{ is constant on } e\}$,
and define the energy
$E_{Q,n}(\sigma)=\sum_{e\in\mathcal C_Q(n)}\delta_e(\sigma)$.
Identifying \(\sigma\) with the corresponding two-coloring, we have
\(E_{Q,n}(\sigma)=M_Q(\sigma)\).  Hence minimizing monochromatic
\(Q\)-constellations is exactly the ground-state problem for the Hamiltonian
\(E_{Q,n}\).

Under the uniformly random coloring, each hyperedge is monochromatic with
probability \(2^{-k}\).  Hence the random-coloring main term for the energy is
$2^{-k}|\mathcal C_Q(n)|
=
\gamma_{\rand}(Q)n^2+O_Q(n)$.
Define the centered energy
$E_{Q,n}^{\mathrm{cent}}(\sigma)
=
E_{Q,n}(\sigma)-2^{-k}|\mathcal C_Q(n)|$.

Theorem~\ref{thm:main} gives the following ground-state bound.

\begin{proposition}\label{prop:spin-ground-state}
There exists a constant \(\delta_Q>0\) such that
\[
\min_{\sigma\in\{-1,+1\}^{[n]}}
E_{Q,n}^{\mathrm{cent}}(\sigma)
\le
-\delta_Q n^2+O_Q(n).
\]
Equivalently,
\[
\min_{\sigma\in\{-1,+1\}^{[n]}}E_{Q,n}(\sigma)
\le
\bigl(\gamma_{\rand}(Q)-\delta_Q\bigr)n^2+O_Q(n).
\]
\end{proposition}

\begin{proof}
The second inequality is exactly Theorem~\ref{thm:main}, written in the spin
notation.  Subtracting
\[
2^{-k}|\mathcal C_Q(n)|=\gamma_{\rand}(Q)n^2+O_Q(n)
\]
gives the centered form.
\end{proof}

The centered Hamiltonian has an even-spin expansion.  For a fixed hyperedge
\(e\),
\[
\delta_e(\sigma)
=
\prod_{i\in e}\frac{1+\sigma_i}{2}
+
\prod_{i\in e}\frac{1-\sigma_i}{2}
=
\frac{1}{2^k}
\sum_{\substack{S\subseteq e\\ |S|\ \mathrm{even}}}
\prod_{i\in S}\sigma_i.
\]
Therefore
\[
E_{Q,n}^{\mathrm{cent}}(\sigma)
=
\frac{1}{2^k}
\sum_{e\in\mathcal C_Q(n)}
\sum_{\substack{S\subseteq e\\ |S|\ \mathrm{even}\\ |S|\ge2}}
\prod_{i\in S}\sigma_i.
\]
Thus the centered energy is an even multi-spin Hamiltonian with deterministic
arithmetic couplings.  Proposition~\ref{prop:spin-ground-state} says that its
ground-state energy lies below the random-coloring main term by order \(n^2\).
This spin-system consequence is deterministic: the couplings are induced by the
arithmetic incidence structure of \(Q\)-constellations, rather than by a random
graph or random hypergraph.

\section*{Acknowledgements}
This work was supported by the National Natural Science Foundation of China
(Nos. 12471329 and 12061059).

\end{document}